\documentclass{amsart}
\usepackage{amssymb}
\usepackage{latexsym}
\usepackage{euscript}

\def\sD{{\mathfrak D}}      
   \def\sH{{\mathfrak H}}

      \def\dR{{\mathbb R}}

\def\wt#1{{{\widetilde #1} }}

\def\bm\chi{\mbox{\boldmath$\chi$}}

\def\ker{{\rm ker\,}}
\def\ran{{\rm ran\,}}
\def\cran{{\rm \overline{ran}\,}}
\def\dom{{\rm dom\,}}

\def\clos{{\rm clos\,}}

\let\xker=\ker \def\ker{{\xker\,}}

\def\sign{{\rm sign\,}}

\def\uphar{{\upharpoonright\,}}

\newcommand\tA{\widetilde{A}}

\def\Ext{{\rm Ext\,}}

\newcommand {\sk}[3]{\left#1#2\right#3}  

\renewcommand {\k}{\kappa}

\newtheorem{theorem}{Theorem}[section]

\newtheorem{lemma}[theorem]{Lemma}

\theoremstyle{definition}

\newtheorem{definition}[theorem]{Definition}

\numberwithin{equation}{section}

\begin{document}

\title[Completion and extension of operators in Kre\u{\i}n spaces.]
{Completion of operators in Kre\u{\i}n spaces
}
\author{D.~Baidiuk}
\date{08.02.2016}
\address{Department of Mathematics and Statistics \\
University of Vaasa \\
P.O. Box 700, 65101 Vaasa \\
Finland} \email{dbaidiuk@uwasa.fi}
%
%





\begin{abstract}
A generalization of the well-known results of M.G. Kre\u{\i}n about the description of selfadjoint contractive extension of a hermitian contraction is obtained. This generalization concerns the situation, where the selfadjoint operator $A$ and extensions $\wt A$ belong to a Kre\u{\i}n space or a Pontryagin space  and their defect operators are allowed to have a fixed number of negative eigenvalues. Also a result of Yu.L. Shmul'yan on completions of nonnegative block operators is generalized for block operators with a fixed number of negative eigenvalues in a Kre\u{\i}n space.

This paper is a natural continuation of S. Hassi's and author's paper \cite{BH2015}.
\end{abstract}

\maketitle

\section{Introduction}
In 1947 M.G.~Kre\u{\i}n published one of his famous papers
\cite{Kr47} on a description of a nonnegative selfadjoint extensions of a densely defined nonnegative
operator $A$ in a Hilbert space. Namely, all nonnegative selfadjoint
extensions $\wt A$ of $A$ can be characterized by the following two
inequalities:
\[
 (A_F+a)^{-1}\le (\tA+a)^{-1}\le (A_K+a)^{-1},
 \quad a>0,
\]
where the Friedrichs (hard) extension $A_F$ and the Kre\u{\i}n-von Neumann
(soft) extension $A_K$ of $A$. He proved these results by transforming the problems the study of contractive operators.

The first result of the present paper is a generalization of a
result due to Shmul'yan \cite{S59} on completions of nonnegative
block operators where the result was applied for introducing
so-called Hellinger operator integrals.  This result was extended in
 \cite{BH2015} for block operators in a Hilbert space by allowing a fixed number
of negative eigenvalues. In Section 2 this result
is further extended to block operators which act in a Kre\u{\i}n space.

In paper \cite{BH2015} we studied classes of ``quasi-contractive''
symmetric operators $T_1$ allowing a finite number of negative
eigenvalues for  the associated defect operator $I-T_1^*T_1$, i.e.,
$\nu_-(I-T_1^*T_1)<\infty$ as well as ``quasi-nonnegative''
operators $A$ with $\nu_-(A)<\infty$ and the existence and
description of all possible selfadjoint extensions $T$ and $\wt A$
of them which preserve the given negative indices
$\nu_-(I-T^2)=\nu_-(I-T_1^*T_1)$ and $\nu_-(\wt A)=\nu_-(A)$, and
proved precise analogs of the above mentioned results of
M.G.~Kre\u{\i}n under a minimality condition on the negative indices
$\nu_-(I-T_1^*T_1)$ and $\nu_-(A)$, respectively. It was an
unexpected fact that when there is a solution then the solution set
still contains a minimal solution and a maximal solution which then
describe the whole solution set via two operator inequalities, just
as in the original paper of M.G.~Kre\u{\i}n. In this paper analogous
results are established for "quasi-contractive" operators acting in
a Kre\u{\i}n space; see Theorems \ref{T:contr2}, \ref{T:contr}.

In Section 4 a first Kre\u{\i}n space analog of completion problem
is formulated and a description of its solutions is found. Namely,
we consider classes of "quasi-contractive" symmetric operators $T_1$
in a Kre\u{\i}n space  with $\nu_-(I-T_1^*T_1)<\infty$ and we
describe all possible selfadjoint (in the Kre\u{\i}n space sense)
extensions $T$ of $T_1$ which preserve the given negative index
$\nu_-(I-T^*T)=\nu_-(I-T_1^*T_1)$. This problem is close to the
completion problem studied in \cite{BH2015} and has a similar
description for its solutions. For further history behind this
problem see also \cite{ACS2006, AG82, ACG87, CG89, CG92,DaKaWe,
Drit90, DritRov90,HMS04, KM1, Kr44, Kr46, ShYa}.

The main result of the present paper is proved in Section 5.
Namely, we consider classes of
"quasi-contractive" symmetric
operators $T_1$ in a Kre\u{\i}n space $(\sH,J)$ with
\begin{equation}\label{indform}
\nu_-[I-T_1^{[*]}T_1]:=\nu_-(J(I-T^{[*]}_1T_1))<\infty
\end{equation}
and we establish a solvability criterion and a description of all
possible selfadjoint extensions $T$ of $T_1$ (in the Kre\u{\i}n
space sense) which preserve the given negative index
$\nu_-[I-T^{[*]}T]=\nu_-[I-T_1^{[*]}T_1]$. It should be pointed out
that in this more general setting the descriptions involve so-called
link operator $L_T$ which was introduced by Arsene, Constantintscu
and Gheondea in \cite{ACG87} (see also \cite{AG82, CG89, CG92,
LT87}).

\section{A completion problem for block operators in Kre\u{\i}n spaces}
By definition the modulus $|C|$ of a closed operator $C$ is the
nonnegative selfadjoint operator $|C|=(C^*C)^{1/2}$. Every closed
operator admits a polar decomposition $C=U|C|$, where $U$ is a
(unique) partial isometry with the initial space $\cran |C|$ and the
final space $\cran C$, cf. \cite{Kato}. For a selfadjoint operator
$H=\int_{\dR} t\, dE_t$ in a Hilbert space $\sH$ the partial
isometry $U$ can be identified with the signature operator, which
can be taken to be unitary: $J=\sign(H)=\int_{\dR}\,\sign(t)\,dE_t$,
in which case one should define $\sign(t)=1$ if $t\ge 0$ and
otherwise $\sign(t)=-1$.

Let $\mathcal{H}$ be a Hilbert space, and let $J_{\mathcal{H}}$ be
 a signature operator in it, i.e.,
$J_{\mathcal{H}}=J_{\mathcal{H}}^*=J_{\mathcal{H}}^{-1}$.
 We
interpret the space $\mathcal{H}$ as a Kre\u{\i}n space
$(\mathcal{H},J_{\mathcal{H}})$ (see \cite{AIbook, Bognarbook}) in
which the indefinite scalar product is defined by the equality
$$
[\varphi,
\psi]_{\mathcal{H}}=
 (J_{\mathcal{H}} \varphi, \psi)_{\mathcal{H}}.
$$
Let us introduce a partial ordering for selfadjoint Kre\u{\i}n space operators.
For selfadjoint operators $A$ and $B$ with the
same domains $A\ge_JB$ if and only if $[(A-B)f,f]\ge 0$ for all
$f\in\dom A$. If not otherwise indicated the word "smallest" means the
smallest operator in the sense of this partial ordering.

Consider a bounded incomplete block operator
\begin{equation}\label{A0}
 A^0=
  \begin{pmatrix}
   A_{11}&A_{12}\\
   A_{21}&\ast
  \end{pmatrix}
  \begin{pmatrix}
   (\sH_1,J_1)\\
   (\sH_2,J_2)
  \end{pmatrix}
  \to
  \begin{pmatrix}
   (\sH_1,J_1)\\
   (\sH_2,J_2)
  \end{pmatrix}
\end{equation}
in the Kre\u{\i}n space $\sH=(\sH_1\oplus\sH_2,J)$, where
$(\sH_1,J_1)$ and $(\sH_2,J_2)$ are
  Kre\u{\i}n spaces with fundamental symmetries $J_1$ and $J_2$, respectively, and
$J=
\begin{pmatrix}
J_1&0\\
0&J_2
\end{pmatrix}$.
\begin{theorem}\label{T:1}
Let $\sH=(\sH_1\oplus\sH_2,J)$ be an orthogonal decomposition of the
Kre\u{\i}n space $\sH$ and let $A^0$ be an incomplete block operator
of the form \eqref{A0}. Assume that $A_{11}=A_{11}^{[*]}$ and
$A_{21}=A_{12}^{[*]}$ are bounded, the numbers of negative squares
of the quadratic form $[A_{11}f,f]$ $(f\in \dom A_{11})$
$\nu_-[A_{11}]:=\nu_-(J_1 A_{11})=\k<\infty$, where
$\k\in\mathbb{Z}_+$, and let us introduce $J_{11}:=\sign(J_1A_{11})$
the (unitary) signature operator of $J_1A_{11}$. Then:
\begin{enumerate}\def\labelenumi{\rm(\roman{enumi})}
\item
 There exists a completion $A\in[(\sH,J)]$ of $A^0$ with some
 operator $A_{22}=A_{22}^{[*]}\in[(\sH_2,J_2)]$ such that
 $\nu_-[A]=\nu_-[A_{11}]=\k$ if and only if
 \begin{equation*}\label{E:1}
  \ran J_1A_{12}\subset\ran |A_{11}|^{1/2}.
 \end{equation*}

\item
In this case the operator $S=|A_{11}|^{[-1/2]}J_1A_{12}$, where
$|A_{11}|^{[-1/2]}$ denotes the (generalized) Moore-Penrose inverse
of $|A_{11}|^{1/2}$, is well defined and
$S\in[(\sH_2,J_2),(\sH_1,J_1)]$. Moreover, $S^{[*]}J_1J_{11}S$ is
the "smallest" operator in the solution set
\begin{equation*}\label{E:sol}
 \mathcal{A}:=\sk\{{A_{22}=A_{22}^{[*]}\in[(\sH_2,J_2)]:\, A=(A_{ij})_{i,j=1}^{2}:\nu_-[A]=\k
 }\}
\end{equation*}
and this solution set admits a description
\[
  \mathcal{A}=\sk\{{A_{22}\in[(\sH_2,J_2)]:\, A_{22}=J_2(S^*J_{11}S+Y)=S^{[*]}J_1J_{11}S+J_2Y,\, Y=Y^*\ge 0}\}.
\]
\end{enumerate}
\end{theorem}

\begin{proof}
 Let us introduce a block operator
 $$
  \wt A^0=
   \begin{pmatrix}
    \wt A_{11}&\wt A_{12}\\
    \wt A_{21}&\ast
   \end{pmatrix}
=
   \begin{pmatrix}
     J_1 A_{11}&J_1 A_{12}\\
    J_2 A_{21}&\ast
   \end{pmatrix}.
 $$
 The blocks of this operator satisfy the identities $\wt A_{11}=\wt
 A_{11}^*$, $\wt A_{21}^*=\wt
 A_{12}$ and
 \begin{equation*}
   \begin{split}
     \ran J_1 A_{11}&=\ran \wt A_{11}\subset\ran |\wt A_{11}|^{1/2}=\ran (\wt A^*_{11}\wt A_{11})^{1/4}\\
     &=\ran (A^*_{11} A_{11})^{1/4}=\ran |A_{11}|^{1/2}.
   \end{split}
 \end{equation*}

 Then due to \cite[Theorem 1]{BH2015} a description of all selfadjoint operator completions of $\wt A^0$ admits
  representation $\wt A=\begin{pmatrix}
    \wt A_{11}&\wt A_{12}\\
    \wt A_{21}&\wt A_{22}
    \end{pmatrix}$ with $\wt A_{22}=\wt S^*J_{11} \wt S+Y$, where
     $\wt S=|\wt A_{11}|^{[-1/2]}\wt A_{12}$ and
    $Y=Y^*\ge 0$.

    This yields description for the solutions of the completion problem.
    The set of completions has the form $ A=\begin{pmatrix}
     A_{11}& A_{12}\\
     A_{21}& A_{22}
    \end{pmatrix}$, where
   \begin{equation*}
    \begin{split}
     A_{22}&=J_2\wt A_{22}= J_2
     A_{21}J_1|A_{11}|^{[-1/2]}J_{11}|A_{11}|^{[-1/2]}J_1A_{12}+J_2Y\\
     &=J_2S^*J_{11}S+J_2Y=S^{[*]}J_{1}J_{11}S+J_2Y.\qedhere
    \end{split}
    \end{equation*}
\end{proof}

\section{Some inertia formulas}
\label{sec2}


Some simple inertia formulas are now recalled. The factorization
$H=B^{[*]}EB$ clearly implies that $\nu_\pm[H]\le \nu_\pm[E]$, cf. \eqref{indform}. If
$H_1$ and $H_2$ are selfadjoint operators in a Kre\u{\i}n space,
then
\[
 H_1+H_2=\begin{pmatrix} I \\ I \end{pmatrix}^{[*]} \begin{pmatrix} H_1 & 0 \\ 0 & H_2 \end{pmatrix}
 \begin{pmatrix} I \\ I \end{pmatrix}
\]
shows that $\nu_\pm[H_1+H_2]\le \nu_\pm[H_1]+\nu_\pm[H_2]$. Consider
the selfadjoint block operator $H\in[(\sH_1,J_1)\oplus(\sH_2,J_2)]$,
where $J_i=J_i^*=J_i^{-1},\ (i=1,2)$ of the form
\begin{equation*}\label{H}
 H=H^{[*]}=\begin{pmatrix} A & B^{[*]} \\ B & I \end{pmatrix},
\end{equation*}
 By applying the above mentioned inequalities shows that
\begin{equation}\label{minneg}
 \nu_\pm[A]\le \nu_\pm[A-B^{[*]}B]+\nu_\pm(J_2).
\end{equation}
Assuming that $\nu_-[A-B^*J_2B]$ and $\nu_-(J_2)$ are finite, the
question when $\nu_-[A]$ attains its maximum in \eqref{minneg}, or
equivalently, $\nu_-[A-B^*J_2B]\ge \nu_-[A]-\nu_-(J_2)$ attains its
minimum, turns out to be of particular interest. The next result
characterizes this situation as an application of Theorem~\ref{T:1}.
Recall that if $J_1A=J_A |A|$ is the polar decomposition of $J_1A$, then
one can interpret $\sH_A=(\cran J_1A,J_A)$ as a Kre\u{\i}n space
generated on $\cran J_1A$ by the fundamental symmetry $J_A=\sign(J_1A)$.

\begin{theorem}\label{thmB}
Let $A\in[(\sH_1,J_1)]$ be selfadjoint,
$B\in[(\sH_1,J_1),(\sH_2,J_2)]$, $J_i=J_i^*=J_i^{-1}\in[\sH_i],\
(i=1,2)$, and assume that $\nu_-[A],\nu_-(J_2)<\infty$. If the
equality
\begin{equation*}\label{min}
 \nu_-[A] = \nu_-[A-B^{[*]}B]+\nu_-(J_2)
\end{equation*}
holds, then $\ran J_1B^{[*]}\subset \ran |A|^{1/2}$ and
$J_1B^{[*]}=|A|^{1/2}K$ for a unique operator
$K\in[(\sH_2,J_2),\sH_A]$ which is $J$-contractive: $J_2-K^*J_A K\ge 0$.

Conversely, if $B^{[*]}=|A|^{1/2}K$ for some $J$-contractive
operator $K\in[(\sH_2,J_2),\sH_A]$, then the equality
\eqref{min} is satisfied.
\end{theorem}
\begin{proof}
Assume that \eqref{min} is satisfied. The factorization
\[
 H=\begin{pmatrix} A & B^{[*]} \\ B & I \end{pmatrix}
 = \begin{pmatrix} I & B^{[*]}\\ 0 & I \end{pmatrix}
  \begin{pmatrix} A-B^{[*]} B& 0 \\ 0 & I \end{pmatrix}
   \begin{pmatrix} I  & 0 \\  B & I \end{pmatrix}
\]
shows that $\nu_-[H]=\nu_-[A-B^{[*]} B]+\nu_-(J_2)$, which combined
with the equality \eqref{min} gives $\nu_-[H]=\nu_-[A]$. Therefore,
by Theorem~\ref{T:1} one has $\ran J_1B^{[*]}\subset \ran |A|^{1/2}$
and this is equivalent to the existence of a unique operator $K\in
[(\sH_2,J_2),\sH_A]$ such that $J_1B^{[*]}=|A|^{1/2}K$; i.e.
$K=|A|^{[-1/2]}J_1B^{[*]}$. Furthermore, $K^{[*]}J_1J_{A}K\leq_{J_2}
I$ by the minimality property of $K^{[*]}J_1J_{A}K$ in
Theorem~\ref{T:1}, in other words $K$ is a $J$-contraction.

Converse, if $J_1B^{[*]}=|A|^{1/2}K$ for some $J$-contraction $K\in
[(\sH_2,J_2),\sH_A]$, then clearly $\ran J_1
B^{[*]}\subset\ran |A|^{1/2}$. By Theorem~\ref{T:1} the completion
problem for $H^{0}$ has solutions with the minimal solution
$S^{[*]}J_1J_{A}S$, where
$$S=|A|^{[-1/2]}J_1B^{[*]}=|A|^{[-1/2]}|A|^{1/2}K=K.$$
 Furthermore, by
$J$-contractivity of $K$ one has ${K^{[*]}J_1J_{A}K\le_{J_2} I}$, i.e. $I$
is also a solution and thus $\nu_-[H]=\nu_-[A]$ or, equivalently,
the equality \eqref{min} is satisfied.
\end{proof}

\section{A pair of completion problems in a Kre\u{\i}n space}

In this section we introduce and describe the solutions of a Kre\u{\i}n space version of a completion problem that was treated in \cite{BH2015}.

Let $(\sH_i,(J_i\cdot,\cdot))$ and $(\sH,(J\cdot,\cdot))$ be Kre\u{\i}n spaces, where $\sH=\sH_1\oplus\sH_2$,$J=\begin{pmatrix}
  J_1&0\\
  0&J_2
\end{pmatrix}$, and $J_i$ are fundamental
symmetries $(i=1,2)$, let $T_{11}=T_{11}^{[*]}\in[(\sH_1,J_1)]$ be
an operator such that $\nu_-(I-T_{11}^*T_{11})=\kappa<\infty$.
Denote $\wt T_{11}=J_1T_{11}$, then $\wt T_{11}=\wt T_{11}^*$ in the
Hilbert space $\sH_1$. Rewrite $\nu_-(I-T_{11}^*T_{11})=\nu_-(I-\wt
T_{11}^2)$. Denote
\begin{equation}\label{E:J}
 J_+=\sign(I-\wt T_{11}),\
 J_-=\sign(I+\wt T_{11}), \text{ and }
J_{11}=\sign(I-\wt T_{11}^2),
\end{equation}
and let $\kappa_+=\nu_-(J_+)$ and $\kappa_-=\nu_-(J_-)$. It is easy
to get that $J_{11}=J_-J_+=J_+J_-$. Moreover, there is an equality
$\kappa=\kappa_- +\kappa_+$ (see \cite[Lemma 5.1]{BH2015}). We recall
the results for the operator $\wt T_{11}$ from the paper
\cite{BH2015} and after that reformulate them for the operator
$T_{11}$.
We recall completion problem and its solutions that was investigated in a Hilbert space setting in  \cite{BH2015}.
The problem concerns the existence and a description of
selfadjoint operators $\wt T$ such that $\wt A_+=I+\wt T$ and $\wt
A_-=I-\wt T$ solve the corresponding completion problems
\begin{equation}\label{E:A1}
\wt A_{\pm}^0=
\begin{pmatrix}
 I\pm \wt T_{11}&\pm \wt T_{21}^*\\
 \pm \wt T_{21}&\ast
\end{pmatrix},
\end{equation}
under \emph{minimal index conditions} $\nu_-(I+\wt T)=\nu_-(I+\wt
T_{11})$, $\nu_-(I-\wt T)=\nu_-(I-\wt T_{11})$, respectively. The solution set is denoted by $\Ext_{\wt T_1,\kappa}(-1,1)$.

The next theorem gives a general solvability criterion for the
completion problem \eqref{E:A1} and describes all solutions to this
problem.

\begin{theorem}\label{T:contr1}
$($\cite[Theorem 5]{BH2015}$)$ Let $\wt T_1=\begin{pmatrix}\wt
T_{11}\\ \wt T_{21}
\end{pmatrix}:\sH_1\to\begin{pmatrix}\sH_1\\\sH_2\end{pmatrix}$ be a symmetric operator
 with
$\wt T_{11}=\wt T_{11}^*\in[\sH_1]$ and $\nu_-(I-\wt
T_{11}^2)=\kappa<\infty$, and let $J_{11}=\sign(I-\wt T_{11}^2)$.
Then the completion problem for $\wt A_{\pm}^0$ in \eqref{E:A1} has a
solution $I\pm \wt T$ for some $\wt T=\wt T^*$ with $\nu_-(I-\wt
T^2)=\kappa$ if and only if the following condition is satisfied:
\begin{equation}\label{crit1}
 \nu_-(I-\wt T_{11}^2)=\nu_-(I-\wt T_1^*\wt T_1).
\end{equation}
If this condition is satisfied then the following facts hold:
\begin{enumerate}\def\labelenumi{\rm(\roman{enumi})}
\item The completion problems for $\wt A_{\pm}^0$ in
\eqref{E:A1} have minimal solutions $\wt A_\pm$.

\item The operators $\wt T_m:=\wt A_+-I$ and $\wt T_M:=I-\wt A_-\in
\Ext_{\wt T_1,\kappa}(-1,1)$.

\item The operators $\wt T_m$ and $\wt T_M$ have the block form
\begin{equation}\label{E:T1}
\begin{split}
\wt T_m&=
\begin{pmatrix}
 \wt T_{11}&D_{\wt T_{11}}V^*\\
 VD_{\wt T_{11}}&-I+V(I-\wt T_{11})J_{11}V^*
\end{pmatrix},\\
\wt T_M=&
\begin{pmatrix}
 \wt T_{11}&D_{\wt T_{11}}V^*\\
 VD_{\wt T_{11}}&I-V(I+\wt T_{11})J_{11}V^*
\end{pmatrix},
\end{split}
\end{equation}
where $D_{\wt T_{11}}:=|I-\wt T_{11}^2|^{1/2}$ and $V$ is given by
$V:=\clos(\wt T_{21}D_{\wt T_{11}}^{[-1]})$.

\item The operators $\wt T_m$ and $\wt T_M$ are extremal extensions of $\wt T_1$:
\begin{equation*}\label{E:Ext1}
\wt T\in\Ext_{\wt T_1,\kappa}(-1,1)\  \text{  iff  }\  \wt T=\wt
T^*\in[\sH],\quad \wt T_m\leq \wt T\leq \wt T_M.
\end{equation*}

\item The operators $\wt T_m$ and $\wt T_M$ are connected via
\begin{equation*}\label{E:5_1}
(-\wt T)_m=-\wt T_M, \quad (-\wt T)_M=-\wt T_m.
\end{equation*}
\end{enumerate}
\end{theorem}
For what follows it is convenient to reformulate the above theorem
in a Kre\u{\i}n space setting. Consider the Kre\u{\i}n space $(\sH,J)$ and a selfadjoint operator
$T$ in this space. Now the problem concerns selfadjoint operators $A_+=I+T$ and $A_-=I-T$ in the Kre\u{\i}n space $(\sH,J)$ that solve the
completion problems
\begin{equation}\label{E:A2}
A_{\pm}^0=
\begin{pmatrix}
 I\pm T_{11}&\pm T_{21}^{[*]}\\
 \pm T_{21}&\ast
\end{pmatrix},
\end{equation}
under \emph{minimal index conditions}
$\nu_-(I+JT)=\nu_-(I+J_1T_{11})$ and
$\nu_-(I-JT)=\nu_-(I-J_1T_{11})$, respectively. The set of solutions
$T$ to the problem \eqref{E:A2} will be denoted by
$\Ext_{J_2T_1,\kappa}(-1,1)$.

Denote
\begin{equation}\label{Tcol2}
T_1=\begin{pmatrix}T_{11}\\  T_{21}
\end{pmatrix}:(\sH_1,J_1)\to\begin{pmatrix}(\sH_1,J_1)\\(\sH_2,J_2)\end{pmatrix},
\end{equation}
so that $T_1$ is symmetric (nondensely defined) operator in the Kre\u{\i}n space
$[(\sH_1,J_1)]$, i.e. $T_{11}=T_{11}^{[*]}$.
\begin{theorem}\label{T:contr2}
Let $T_1$ be a symmetric operator in a Kre\u{\i}n space sense as in
\eqref{Tcol2} with $T_{11}=T_{11}^{[*]}\in[(\sH_1,J_1)]$ and
$\nu_-(I-T_{11}^*T_{11})=\kappa<\infty$, and let
$J=\sign(I-T_{11}^*T_{11})$. Then the completion problems for
$A_{\pm}^0$ in \eqref{E:A2} have a solution $I\pm T$ for some
$T=T^{[*]}$ with $\nu_-(I-T^*T)=\kappa$ if and only if the following
condition is satisfied:
\begin{equation}\label{crit2}
 \nu_-(I-T_{11}^*T_{11})=\nu_-(I-T_1^*T_1).
\end{equation}
If this condition is satisfied then the following facts hold:
\begin{enumerate}\def\labelenumi{\rm(\roman{enumi})}
\item The completion problems for $A_{\pm}^0$ in
\eqref{E:A2} have "minimal"($J_2$-minimal) solutions $A_\pm$.

\item The operators $T_m:=A_+-J$ and $T_M:=J-A_-\in
\Ext_{J_2T_1,\kappa}(-1,1)$.

\item The operators $T_m$ and $T_M$ have the block form
\begin{equation}\label{E:T2}
\begin{split}
&T_m=
\begin{pmatrix}
 T_{11}&J_1D_{T_{11}}V^*\\
 J_2VD_{T_{11}}&-J_2+J_2V(I-J_1T_{11})J_{11}V^*
\end{pmatrix},\\
&T_M=
\begin{pmatrix}
 T_{11}&J_1D_{T_{11}}V^*\\
 J_2VD_{T_{11}}&J_2-J_2V(I+J_1T_{11})J_{11}V^*
\end{pmatrix},
\end{split}
\end{equation}
where $D_{T_{11}}:=|I-T_{11}^*T_{11}|^{1/2}$ and $V$ is given by
$V:=\clos(J_2T_{21}D_{T_{11}}^{[-1]})$.

\item The operators $T_m$ and $T_M$ are $J_2$-extremal extensions of $T_1$:
\begin{equation*}\label{E:Ext2}
T\in \Ext_{J_2T_1,\kappa}(-1,1)\  \text{  iff  }\
T=T^{[*]}\in[(\sH,J)],\quad T_m\leq_{J_2} T\leq_{J_2} T_M.
\end{equation*}

\item The operators $T_m$ and $T_M$ are connected via
\begin{equation*}\label{E:5_2}
(-T)_m=-T_M, \quad (-T)_M=-T_m.
\end{equation*}
\end{enumerate}
\end{theorem}
\begin{proof}
  The proof is obtained by systematic use of the equivalence that $T$ is a selfadjoint operator in a Kre\u{\i}n space if and only if $\wt T$ is a selfadjoint in a Hilbert space. In particular, $T$ gives solutions to the completion problems \eqref{E:A2} if and only if $\wt T$ solves the completion problems \eqref{E:A2}. In view of
  $$
  I-T_{11}^*T_{11}=I-T_{11}^*JJT_{11}=I-\wt T_{11}^2,
  $$
  we are getting formula \eqref{crit2} from \eqref{crit1}.
  Then formula \eqref{E:T2} follows by multiplying the operators in \eqref{E:T1} by the fundamental symmetry.
\end{proof}

\section{Completion problem in a Pontryagin space}
 \begin{subsection}{Defect operators and link operators}

 Let $(\sH,(\cdot,\cdot))$  be a Hilbert space and let $J$  be a symmetry in
$\sH$, i.e. $J=J^*=J^{-1}$, so that $(\sH,(J\cdot,\cdot))$, becomes
a Pontryagin space. Then associate with $T\in[\sH]$ the
corresponding defect and signature operators
\[
 D_T=|J-T^*JT|^{1/2},\quad J_T=\sign(J-T^*JT), \quad \sD_T=\cran D_T,
\]
where the so-called defect subspace $\sD_T$ can be considered as a
Pontryagin space with the fundamental symmetry $J_T$. Similar
notations are used with $T^*$:
\[
 D_{T^*}=|J-TJT^*|^{1/2},\quad J_{T^*}=\sign(J-TJT^*), \quad \sD_{T^*}=\cran D_{T^*}.
\]
By definition $J_TD_T^2=J-T^*JT$ and $J_TD_T=D_TJ_T$ with analogous
identities for $D_{T^*}$ and $J_{T^*}$. In addition,
\begin{equation*}\label{eqC3}
 (J-T^*JT)JT^*=T^*J(J-TJT^*), \,
 (J-TJT^*)JT=TJ(J-T^*JT).
\end{equation*}

Recall that $T\in[\sH]$ is said to be a $J$-contraction if
$J-T^*JT\ge 0$, i.e. $\nu_-(J-T^*JT)=0$. If, in addition, $T^*$ is a
$J$-contraction, $T$ is termed as a $J$-bicontraction.

For the following consideration an indefinite version of the
commutation relation of the form $TD_T=D_{T^*}T$ is needed; these
involve so-called link operators introduced in \cite[Section
4]{ACG87} (see also \cite{BH2015}).

\begin{definition}\label{linkoper}
There exist unique operators $L_T\in [\sD_T,\sD_{T^*}]$ and
$L_{T^*}\in [\sD_{T^*},\sD_T]$ such that
\begin{equation}\label{E:linkoper}
  D_{T^*}L_T=TJ D_T\uphar \sD_T, \quad
 D_T L_{T^*}=T^*J D_{T^*}\uphar \sD_{T^*};
\end{equation}
in fact, $L_T=D_{T^*}^{[-1]}TJ D_T\uphar \sD_T$ and
$L_{T^*}=D_T^{[-1]}T^*J D_{T^*}\uphar \sD_{T^*}$.
\end{definition}

The following identities can be obtained with direct calculations;
see \cite[Section~4]{ACG87}:
\begin{equation}\label{Link2}
\begin{array}{c}
  L_T^* J_{T^*}\uphar \sD_{T^*}=J_{T}L_{T^*};\\
  (J_T-D_TJD_T)\uphar \sD_T=L_T^*J_{T^*}L_T;\\
  (J_{T^*}-D_{T^*}JD_{T^*})\uphar \sD_{T^*}=L_{T^*}^* J_T L_{T^*}.
\end{array}
\end{equation}

Now let $T$ be selfadjoint in Pontryagin space $(\sH,J)$, i.e.
$T^*=JTJ$. Then connections between $D_{T^*}$ and $D_T$, $J_{T^*}$
and $J_T$, $L_{T^*}$ and $L_T$ can be established.
\begin{lemma}\label{L:DTstar}
 Assume that $T^*=JTJ$. Then $D_T=|I-T^2|^{1/2}$ and the following equalities hold:
  \begin{equation}\label{E:DTstar}
   D_{T^*}=JD_TJ,
  \end{equation}
  in particular,
  \begin{equation*}\label{E:domDTstar}
   \sD_{T^*}=J\sD_T\text{ and }\sD_{T}=J\sD_{T^*};
  \end{equation*}
  \begin{equation}\label{E:JTstar}
   J_{T^*}=JJ_TJ;
  \end{equation}
  \begin{equation}\label{E:LTstar}
   L_{T^*}=JL_TJ.
  \end{equation}
\end{lemma}
\begin{proof}
 The defect operator of $T$ can be calculated by the formula
 $$
  D_T=\sk({\sk({I-(T^*)^2})JJ(I-T^2)})^{1/4}=\sk({\sk({I-(T^*)^2})(I-T^2)})^{1/4}.
 $$
 Then
 $$
  D_{T^*}=\sk({J\sk({I-(T^*)^2})(I-T^2)J})^{1/4}=J\sk({\sk({I-(T^*)^2})(I-T^2)})^{1/4}J=JD_TJ
 $$
 i.e. \eqref{E:DTstar} holds. This implies
 $$
 J\sD_{T^*}\subset\sD_{T}\text{ and }J\sD_T\subset\sD_{T^*}.
 $$
 Hence from the last two formulas we get
 $$
 \sD_{T^*}=J(J\sD_{T^*})\subset J\sD_T\subset\sD_{T^*}
 $$
 and similarly
 $$
 \sD_{T}=J(J\sD_{T})\subset J\sD_{T^*}\subset\sD_{T}.
 $$
 The formula
 $$
  J_TD_T^2=J-T^*JT=J(J-TJT^*)J=JJ_{T^*}D_{T^*}^2J=JJ_{T^*}JD_T^2JJ=JJ_{T^*}JD_T^2
 $$
 yields the equation \eqref{E:JTstar}.

 The relation \eqref{E:LTstar} follows from
 $$
  D_TL_{T^*}=T^*JD_{T^*}\uphar\sD_{T^*}=JTJD_TJ\uphar\sD_{T^*}=JD_{T^*}L_TJ=D_TJL_TJ.\qedhere
 $$
\end{proof}
\end{subsection}
\begin{subsection}{Lemmas on negative indices of certain block operators}
The first two lemmas are of preparatory nature for the last two lemmas, which are used for the proof of the
main theorem.
\begin{lemma}\label{L:JT}
Let $\begin{pmatrix} J&T\\T&J\end{pmatrix}:\begin{pmatrix}\sH\\
\sH\end{pmatrix}\to\begin{pmatrix}\sH\\ \sH\end{pmatrix}$ be a
selfadjoint operator in the Hilbert space $\sH^2=\sH\oplus\sH$.
Then
$$
\sk|{\begin{pmatrix} J&T\\T&J\end{pmatrix}}|^{1/2}=U\begin{pmatrix}
|J+T|^{1/2}&0\\0&|J-T|^{1/2}\end{pmatrix}U^*,
$$
where $U=\frac{1}{\sqrt{2}}\begin{pmatrix} I&I\\I&-I\end{pmatrix}$ is a unitary operator.

\end{lemma}
\begin{proof}
It is easy to check that

\begin{equation}\label{Jpm}
\begin{pmatrix} J&T\\T&J\end{pmatrix}=U\begin{pmatrix}
J+T&0\\0&J-T\end{pmatrix}U^*.
\end{equation}
Then by taking the modulus one gets
$$
\sk|{\begin{pmatrix} J&T\\T&J\end{pmatrix}}|^2=\sk({\begin{pmatrix}
J&T\\T&J\end{pmatrix}^*\begin{pmatrix}
J&T\\T&J\end{pmatrix}})=U\begin{pmatrix}
|J+T|^2&0\\0&|J-T|^2\end{pmatrix}U^*.
$$
The last step is to extract the square roots (twice) from the both sides of
the equation:
$$
\sk|{\begin{pmatrix} J&T\\T&J\end{pmatrix}}|^{1/2}=U\begin{pmatrix}
|J+T|^{1/2}&0\\0&|J-T|^{1/2}\end{pmatrix}U^*.
$$
The right hand side can be written in this form because $U$ is
 unitary.
\end{proof}

\begin{lemma}\label{LJT2}
 Let $T=T^*\in \sH$ be a selfadjoint operator in a Hilbert space $\sH$ and let $J=J^*=J^{-1}$ be a fundamental symmetry in $\sH$ with $\nu_-(J)<\infty$.
 Then
 \begin{equation}\label{E:J2}
  \nu_-(J-TJT)+\nu_-(J)=\nu_-(J-T)+\nu_-(J+T).
 \end{equation}
 In particular, $\nu_-(J-TJT)<\infty$ if and only if $\nu_-(J\pm T)<\infty$.
\end{lemma}
 \begin{proof}
  Consider block operators
  $\begin{pmatrix} J&T\\T&J\end{pmatrix}$ and $\begin{pmatrix}
  J+T&0\\0&J-T\end{pmatrix}$. Equality  \eqref{Jpm} yields \- $\nu_-\begin{pmatrix} J&T\\T&J\end{pmatrix}=\nu_-\begin{pmatrix} J+T&0\\0&J-T\end{pmatrix}$. The negative index of $\begin{pmatrix}
  J+T&0\\0&J-T\end{pmatrix}$ equals $\nu_-(J-T)+\nu_-(J+T)$ and the negative index of $\begin{pmatrix}
  J&T\\T&J\end{pmatrix}$ is easy to find by using the equality
  \begin{equation}\label{E:JT2}
   \begin{pmatrix} J&T\\T&J\end{pmatrix}=\begin{pmatrix} I&0\\TJ&I\end{pmatrix}\begin{pmatrix} J&0\\0&J-TJT\end{pmatrix}\begin{pmatrix} I&JT\\0&I\end{pmatrix}.
 \end{equation}
  Then one gets \eqref{E:J2}.
\end{proof}

Let $(\sH_i,(J_i\cdot,\cdot))$ $(i=1,2)$ and
$(\sH,(J\cdot,\cdot))$ be Pontryagin spaces, where
$\sH=\sH_1\oplus\sH_2$ and $J=\begin{pmatrix}   J_1&0\\0&J_2
\end{pmatrix}$. Consider an operator
$T_{11}=T_{11}^{[*]}\in[(\sH_1,J_1)]$ such that
$\nu_-[I-T_{11}^2]=\kappa<\infty$; see \eqref{indform}. Denote $\wt T_{11}=J_1T_{11}$,
then $\wt T_{11}=\wt T_{11}^*$ in the Hilbert space $\sH_1$. Rewrite
$$
\nu_-[I-T_{11}^2]=\nu_-(J_1(I-T_{11}^2))=\nu_-(J_1-\wt T_{11}J_1\wt T_{11})=\nu_-((J_1-\wt
T_{11})J_1(J_1+\wt T_{11})).
$$
Furthermore, denote
\begin{equation}\label{E:J}
\begin{split}
 &J_+=\sign(J_1(I-T_{11}))=\sign(J_1-\wt T_{11}),\\
 &J_-=\sign(J_1(I+T_{11}))=\sign(J_1+\wt T_{11}),\\
&J_{11}=\sign(J_1(I-T_{11}^2))
\end{split}
\end{equation}
and let $\kappa_+=\nu_-[I-T_{11}]$ and $\kappa_-=\nu_-[I+T_{11}]$. Notice that
$|I\mp T_{11}|=|J_1\mp \wt T_{11}|$ and one has polar decompositions
\begin{equation}\label{E:Jb}
I\mp T_{11}=J_1J_\pm|I\mp T_{11}|.
\end{equation}
\begin{lemma}\label{L:wT}
Let $T_{11}=T_{11}^{[*]}\in[(\sH_1,J_1)]$ and $ T=
\begin{pmatrix}  T_{11}& T_{12}\\ T_{21}& T_{22}\end{pmatrix}\in [(\sH,J)] $ be a
selfadjoint extension of the operator $ T_{11}$ with $\nu_-[I\pm T_{11}]<\infty$ and $\nu_-(J)<\infty$. Then the following
statements
\begin{enumerate}
\item[(i)]
$\nu_-[I\pm T_{11}]=\nu_-[I\pm T];$
\end{enumerate}
\begin{enumerate}
\item[(ii)]
$\nu_-[I- T^2]=\nu_-[I-T_{11}^2]-\nu_-(J_2);$
\item[(iii)]
  $\ran J_1T_{21}^{[*]}\subset \ran|I\pm T_{11}|^{1/2}$
\end{enumerate}
are connected by the implications
$(i)\Leftrightarrow(ii)\Rightarrow(iii)$.
\end{lemma}
\begin{proof}
The Lemma can be formulated in an equivalent way for the Hilbert
space operators: the
block operator $\wt T=JT= \begin{pmatrix} \wt T_{11}&\wt T_{12}\\\wt
T_{21}&\wt T_{22}\end{pmatrix}$ is a selfadjoint extension of  $\wt T_{11}=\wt T_{11}^*\in[\sH_1]$. Then the following statements
\begin{enumerate}
\item[(i')]
$\nu_-(J_1\pm\wt T_{11})=\nu_-(J\pm\wt T)$
\item[(ii')]
$\nu_-(J-\wt TJ\wt T)=\nu_-(J_1-\wt T_{11}J_1\wt
T_{11})-\nu_-(J_2);$
\item[(iii')]
  $\ran \wt T_{12}\subset\ran|J_1\pm\wt
T_{11}|^{1/2}$
\end{enumerate}
are connected by the implications $(i')\Leftrightarrow(ii')\Rightarrow(iii')$.

 Hence it's  sufficient to prove this form
of the Lemma.

Let us prove the equivalence $(i')\Leftrightarrow(ii')$. Condition
(ii') is equivalent to
\begin{equation}\label{E:nuT}
\nu_-\begin{pmatrix} J_1&\wt T_{11}\\\wt T_{11}&J_1\end{pmatrix}=\nu_-\begin{pmatrix} J&\wt T\\\wt T&J\end{pmatrix}.
\end{equation}
Indeed, in view of \eqref{E:JT2}
$$
\nu_-\begin{pmatrix} J_1&\wt T_{11}\\\wt
T_{11}&J_1\end{pmatrix}=\nu_-(J_1)+\nu_-(J_1-\wt T_{11}J_1\wt
T_{11})
$$
and
\begin{equation*}
\begin{split}
\nu_-\begin{pmatrix} J&\wt T\\\wt T&J\end{pmatrix}=\nu_-(J)+\nu_-(J-\wt TJ\wt T)=\nu_-(J_1)+\nu_-(J_2)+\nu_-(J-\wt TJ\wt T).
\end{split}
\end{equation*}
By using Lemma \ref{LJT2}, equality \eqref{E:nuT} is equivalent to
\begin{equation}\label{itoii}
\nu_-(J_1-\wt T_{11})+\nu_-(J_1+\wt T_{11})=\nu_-(J-\wt
T)+\nu_-(J+\wt T).
\end{equation}
Hence, $(i')\Rightarrow(ii')$.

 Because $\nu_-(J_1\pm\wt
T_{11})\leq\nu_-(J\pm\wt T)$, then \eqref{itoii} shows that
$(ii')\Rightarrow(i')$.

Now we prove implication $(ii')\Rightarrow(iii')$;the arguments here will be useful also for the proof of Lemma \ref{L:ranran} below. Use a permutation to transform the matrix
in the right hand side of \eqref{E:nuT}:
$$
\nu_-\begin{pmatrix} J&\wt T\\\wt T&J\end{pmatrix}=
\nu_-\begin{pmatrix}
 J_1&0&\wt T_{11}&\wt T_{12}\\
 0&J_2&\wt T_{21}&\wt T_{22}\\
 \wt T_{11}&\wt T_{12}&J_1&0\\
 \wt T_{21}&\wt T_{22}&0&J_2
 \end{pmatrix}=
 \nu_-\begin{pmatrix}
 J_1&\wt T_{11}&0&\wt T_{12}\\
 \wt T_{11}&J_1&\wt T_{12}&0\\
 0&\wt T_{21}&J_2&\wt T_{22}\\
 \wt T_{21}&0&\wt T_{22}&J_2
 \end{pmatrix}.
$$
 Then condition \eqref{E:nuT} implies to the condition
$$
\ran \begin{pmatrix} 0&\wt T_{12}\\\wt T_{12}&0\end{pmatrix}\subset\ran\sk|{\begin{pmatrix} J_1&\wt T_{11}\\\wt T_{11}&J_1\end{pmatrix}}|^{1/2};
$$
(see Theorem \ref{T:1}). By  Lemma \ref{L:JT} the last inclusion can be rewritten as
\begin{equation*}\label{bigran}
\ran \begin{pmatrix} 0&\wt T_{12}\\\wt
T_{12}&0\end{pmatrix}\subset\ran U\begin{pmatrix} |J_1+\wt
T_{11}|^{1/2}&0\\0&|J_1-\wt T_{11}|^{1/2}\end{pmatrix}U^*,
\end{equation*}
where $U=\frac{1}{\sqrt{2}}\begin{pmatrix} I&I\\I&-I\end{pmatrix}$
is a unitary operator. This inclusion is equivalent to
\begin{equation*}\label{bigran2}
\ran U^*\begin{pmatrix} 0&\wt T_{12}\\\wt
T_{12}&0\end{pmatrix}U=\ran\begin{pmatrix} \wt T_{12}&0\\0&-\wt
T_{12}\end{pmatrix}\subset\ran \begin{pmatrix} |J_1+\wt
T_{11}|^{1/2}&0\\0&|J_1-\wt T_{11}|^{1/2}\end{pmatrix}
\end{equation*}
and clearly this is equivalent to condition (iii').

 Note that if $\wt T_{11}$ has a selfadjoint
extension $\wt T$ satisfying (i'). Then by applying Theorem \ref{T:1} (or \cite[Theorem 1]{BH2015}) it yields (iii').
\end{proof}
\begin{lemma}\label{L:ranran}
Let $T_{11}=T_{11}^{[*]}\in[(\sH_1,J_1)]$ be an operator and let
$$T_1=\begin{pmatrix}   T_{11}\\T_{21}
\end{pmatrix}:(\sH_1,J_1)\to \begin{pmatrix}  (\sH_1,J_1)\\(\sH_2,J_2)
\end{pmatrix}$$
 be an extension of $T_{11}$ with $\nu_-[I-T_{11}^2]<\infty$, $\nu_-(J_1)<\infty$, and $\nu_-(J_2)<\infty$. Then for the conditions
\begin{enumerate}
  \item[(i)]
  $\nu_-[I_1-T_{11}^2]=\nu_-[I_1-T_{1}^{[*]}T_{1}]+\nu_-(J_2)$;
  \item[(ii)]
  $\ran J_1T_{21}^{[*]}\subset \ran|I-T_{11}^2|^{1/2};$
  \item[(iii)]
  $\ran J_1T_{21}^{[*]}\subset \ran|I\pm T_{11}|^{1/2}$
  \end{enumerate}
  the implications $(i)\Rightarrow(ii)$ and $(i)\Rightarrow(iii)$ hold.
\end{lemma}
\begin{proof}
 First we prove that (i)$\Rightarrow$(ii). In fact, this follows from Theorem \ref{thmB}  by taking $A=I-T_{11}^2$
  and $B=T_{21}$.

  A proof of (i)$\Rightarrow$(iii) is quite similar to the
  proof used in Lemma \ref{L:wT}. Statement (i) is equivalent the
  following equation:
  \begin{equation*}\label{JtJt}
   \nu_-\begin{pmatrix}
     J_1&\wt T_{11}\\
     \wt T_{11}&J_1
   \end{pmatrix}=\nu_-\begin{pmatrix}
     J&\wt T_1\\
     \wt T_1^*&J_1
   \end{pmatrix}.
  \end{equation*}
  Indeed,
\begin{equation*}
  \begin{split}
   &\nu_-\begin{pmatrix}
     J_1&\wt T_{11}\\
     \wt T_{11}&J_1
   \end{pmatrix}=
   \nu_-\begin{pmatrix}
     J_1&0\\
     0&J_1-\wt T_{11}J_1\wt T_{11}
   \end{pmatrix}\\
   &=\nu_-(J_1-\wt T_{11}J_1\wt T_{11})+
   \nu_-(J_1)<\infty
   \end{split}
  \end{equation*}
  and
  \begin{equation*}
  \begin{split}
   &\nu_-\begin{pmatrix}
     J&\wt T_1\\
     \wt T_1^*&J_1
   \end{pmatrix}=
   \nu_-\begin{pmatrix}
     J&0\\
     0&J_1-\wt T_1^*J\wt T_1
   \end{pmatrix}\\
   &=\nu_-(J_1-\wt T_{11}J_1\wt T_{11}-\wt T_{21}^*J_2\wt T_{21})+
   \nu_-(J_1)+\nu_-(J_2).
   \end{split}
  \end{equation*}
 Due to (i) the right hand sides coincide and then the left hand
sides coincide as well.

Now let us permutate the matrix in the latter equation.
\begin{equation*}
  \begin{split}
   \nu_-\begin{pmatrix}
     J&\wt T_1\\
     \wt T_1^*&J_1
   \end{pmatrix}=
   \nu_-\begin{pmatrix}
     J_1&0&\wt T_{11}\\
     0&J_2&\wt T_{21}\\
     \wt T_{11}&\wt T_{21}^*&J_1
   \end{pmatrix}
   =\nu_-\begin{pmatrix}
     J_1&\wt T_{11}&0\\
     \wt T_{11}&J_1&\wt T_{21}^*\\
     0&\wt T_{21}&J_2
   \end{pmatrix}.
\end{split}
  \end{equation*}
 It follows from \cite[Theorem 1]{BH2015} that the
condition (i) implies the condition
\begin{equation*}
  \ran\begin{pmatrix}0\\\wt T_{21}^*\end{pmatrix}\subset
  \ran\sk|{\begin{pmatrix}J_1&\wt T_{11}\\\wt
  T_{11}&J_1\end{pmatrix}}|^{1/2}=
  \ran U\begin{pmatrix}|J_1+\wt T_{11}|^{1/2}&0\\0&|J_1-\wt
  T_{11}|^{1/2}\end{pmatrix}U^*,
\end{equation*}
where $U=\frac{1}{\sqrt{2}}\begin{pmatrix} I&I\\I&-I\end{pmatrix}$
is a unitary operator (see Lemma \ref{L:JT}). Then, equivalently,
$$
\ran \wt T_{21}^*\subset \ran|J_1\pm \wt T_{11}|^{1/2}.\qedhere
$$
\end{proof}
\end{subsection}
\begin{subsection}{Contractive extensions of contractions with minimal negative indices}
Following to \cite{BH2015, HMS04, KM1} we consider the problem of existence and a description of
selfadjoint operators $T$ in the Pontryagin space $\begin{pmatrix}
(\sH_1,J_1)\\
(\sH_2,J_2)
\end{pmatrix} $
 such that $
A_+=I+T$ and $ A_-=I-T$ solve the corresponding completion
problems
\begin{equation}\label{E:A}
A_{\pm}^0=
\begin{pmatrix}
 I\pm T_{11}&\pm T_{21}^{[*]}\\
 \pm T_{21}&\ast
\end{pmatrix},
\end{equation}
under \emph{minimal index conditions} $\nu_-[I+T]=\nu_-[I+T_{11}]$,
$\nu_-[I-T]=\nu_-[I-T_{11}]$, respectively. Observe, that by
Lemma~\ref{L:wT} the two minimal index conditions above are
equivalent to single condition
$\nu_-[I-T^2]=\nu_-[I-T_{11}^2]-\nu_-(J_2)$.

It is
clear from Theorem~\ref{T:1} that the conditions $\ran
J_1T_{21}^{[*]}\subset \ran|I-T_{11}|^{1/2}$ and $\ran J_1T_{21}^{[*]}\subset
\ran|I+T_{11}|^{1/2}$ are necessary for the existence of solutions;
however as noted already in \cite{BH2015} they are not sufficient even in the Hilbert space setting.


 The next theorem gives a
general solvability criterion for the completion problem \eqref{E:A}
and describes all solutions to this problem. As in the definite
case, there are minimal solutions $A_+$ and $A_-$ which are
connected to two extreme selfadjoint extensions $T$ of
\begin{equation}\label{Tcol}
T_1=\begin{pmatrix} T_{11}\\ T_{21}
\end{pmatrix}:(\sH_1,J_1)\to\begin{pmatrix}(\sH_1,J_1)\\(\sH_2,J_2)\end{pmatrix},
\end{equation}
now with finite negative index
$\nu_-[I-T^2]=\nu_-[I-T_{11}^2]-\nu_-(J_2)>0$. The set of solutions
$T$ to the problem \eqref{E:A} will be denoted by
$\Ext_{T_1,\kappa}(-1,1)_{J_2}$.

\begin{theorem}\label{T:contr}
Let $T_1$ be a symmetric operator as in \eqref{Tcol} with
$T_{11}=T_{11}^{[*]}\in[(\sH_1,J_1)]$ and
$\nu_-[I-T_{11}^2]=\kappa<\infty$, and let
$J_{T_{11}}=\sign(J_1(I-T_{11}^2))$. Then the completion problem for
$A_{\pm}^0$ in \eqref{E:A} has a solution $I\pm T$ for some
$T=T^{[*]}$ with $\nu_-[I-T^2]=\kappa-\nu_-(J_2)$ if and only if the
following condition is satisfied:
\begin{equation}\label{crit}
 \nu_-[I-T_{11}^2]=\nu_-[I-T_1^{[*]}T_1]+\nu_-(J_2).
\end{equation}
If this condition is satisfied then the following facts hold:
\begin{enumerate}\def\labelenumi{\rm(\roman{enumi})}
\item The completion problems for $A_{\pm}^0$ in
\eqref{E:A} have "minimal" solutions $A_\pm$ (for the partial ordering
 introduced in the first section).

\item The operators $T_m:=A_+-I$ and $T_M:=I-A_-\in
\Ext_{T_1,\kappa}(-1,1)_{J_2}$.

\item The operators $T_m$ and $T_M$ have the block form
\begin{equation}\label{E:T}
\begin{split}
&T_m=
\begin{pmatrix}
T_{11}&J_1D_{T_{11}}V^{*}\\
 J_2VD_{T_{11}}&-I+J_2V(I-L^*_TJ_1)J_{11}V^{*}
\end{pmatrix},\\
&T_M=
\begin{pmatrix}
T_{11}&J_1D_{T_{11}}V^{*}\\
 J_2VD_{T_{11}}&I-J_2V(I+L^*_TJ_1)J_{11}V^{*}
\end{pmatrix},
\end{split}
\end{equation}
where $D_{T_{11}}:=|I-T_{11}^2|^{1/2}$ and $V$ is given by
$V:=\clos(J_2T_{21}D_{T_{11}}^{[-1]})$.

\item The operators $T_m$ and $T_M$ are "extremal" extensions of $T_1$:
\begin{equation}\label{E:Ext}
T\in\Ext_{T_1,\kappa}(-1,1)_{J_2}\  \text{  iff  }\
T=T^{[*]}\in[(\sH,J)],\quad T_m\le_{J_2} T\le_{J_2} T_M.
\end{equation}

\item The operators $T_m$ and $T_M$ are connected via
\begin{equation}\label{E:5}
(-T)_m=-T_M, \quad (-T)_M=-T_m.
\end{equation}
\end{enumerate}
\end{theorem}
\begin{proof}
It is easy to see by \eqref{minneg} that $\kappa=\nu_-[I-T_{11}^2]\le
\nu_-[I-T_1^{[*]}T_1]+\nu_-(J_2)\le \nu_-[I-T^2]+\nu_-(J_2)$. Hence
the condition $\nu_-[I-T^2]=\kappa-\nu_-(J_2)$ implies \eqref{crit}.
The sufficiency of this condition is obtained when proving the
assertions (i)--(iii) below.

(i) If the condition \eqref{crit} is satisfied then by using Lemma \ref{L:ranran} one gets the inclusions $\ran J_1
T_{21}^{[*]}\subset \ran|I\pm T_{11}|^{1/2}$, which by
Theorem~\ref{T:1} means that each of the completion problems,
$A_{\pm}^0$ in \eqref{E:A}, is solvable. It follows that the
operators
\begin{equation}\label{E:S pm}
S_-=|I+T_{11}|^{[-1/2]}J_1T_{21}^{[*]},\quad
S_+=|I-T_{11}|^{[-1/2]}J_1T_{21}^{[*]}
\end{equation}
are well defined and they provide the minimal solutions $A_\pm$ to
the completion problems for $A_\pm^0$ in \eqref{E:A}.

(ii) \& (iii) By Lemma \ref{L:ranran} the inclusion $\ran J_1
T_{21}^{[*]}\subset \ran|I- T_{11}^2|^{1/2}$ holds. This
inclusion is equivalent to the existence of a (unique) bounded
operator $V^{*}=D_{T_{11}}^{[-1]}J_1T_{21}^{[*]}$ with $\ker
V\supset \ker D_{T_{11}}$, such that
$J_1T_{21}^{[*]}=D_{T_{11}}V^{*}$. The operators $T_m:=A_+-I$ and
$T_M:=I-A_-$ (see proof of (i)) by using \eqref{E:linkoper}, \eqref{Link2}, and \ref{L:DTstar} can be now rewritten as in
\eqref{E:T}. Indeed, observe that (see Theorem \ref{T:1}, \eqref{E:J}, and \eqref{E:Jb})
\begin{equation*}
\begin{split}
J_2S_-^{*}J_-S_-&=J_2VD_{T_{11}}|I+T_{11}|^{[-1/2]}J_-|I+T_{11}|^{[-1/2]}D_{T_{11}}V^{*}\\
&=J_2VD_{T_{11}}(J_1(I+T_{11}))^{[-1]}D_{T_{11}}V^{*}\\
&=J_2VD_{T_{11}}D_{T_{11}}^{[-1]}(I+L_{T_{11}}^*J_1)^{[-1]}D_{T_{11}}J_1D_{T_{11}}V^*\\
&=J_2V(I+L_{T_{11}}^*J_1)^{[-1]}(J_{11}-L^*_{T_{11}}J_{T^*_{11}}L_{T_{11}})V^*\\
&=J_2V(I+L_{T_{11}}^*J_1)^{[-1]}(J_{11}-(L^*_{T_{11}}J_1)^2J_{11})V^*\\
&=J_2V(I+L_{T_{11}}^*J_1)^{[-1]}(I+L^*_{T_{11}}J_1)(I-L^*_{T_{11}}J_1)J_{11}V^*\\
&=J_2V(I-L^*_{T_{11}}J_1)J_{11}V^*,
\end{split}
\end{equation*}
where the third equality follows from \eqref{E:linkoper} and the fourth from \eqref{Link2}.

And similarly for
\begin{equation*}
\begin{split}
J_2S_+^{*}J_+S_+&=J_2VD_{T_{11}}|I-T_{11}|^{[-1/2]}J_+|I-T_{11}|^{[-1/2]}D_{T_{11}}V^{*}\\
&=J_2VD_{T_{11}}(J_1(I-T_{11}))^{[-1]}D_{T_{11}}V^{*}\\
&=J_2VD_{T_{11}}D_{T_{11}}^{[-1]}(I-L_{T_{11}}^*J_1)^{[-1]}D_{T_{11}}J_1D_{T_{11}}V^*\\
&=J_2V(I-L_{T_{11}}^*J_1)^{[-1]}(J_{11}-L^*_{T_{11}}J_{T^*_{11}}L_{T_{11}})V^*\\
&=J_2V(I-L_{T_{11}}^*J_1)^{[-1]}(J_{11}-(L^*_{T_{11}}J_1)^2J_{11})V^*\\
&=J_2V(I-L_{T_{11}}^*J_1)^{[-1]}(I-L^*_{T_{11}}J_1)(I+L^*_{T_{11}}J_1)J_{11}V^*\\
&=J_2V(I+L^*_{T_{11}}J_1)J_{11}V^*,
\end{split}
\end{equation*}
which implies the representations for $T_m$ and $T_M$ in
\eqref{E:T}. Clearly, $T_m$ and $T_M$ are selfadjoint extensions of
$T_1$, which satisfy the equalities
$$
\nu_-[I+T_{m}]=\kappa_-,\quad \nu_-[I-T_{M}]=\kappa_+.
$$
Moreover, it follows from \eqref{E:T} that
\begin{equation}\label{E:6}
T_M-T_m=
\begin{pmatrix}
 0&0\\
 0&2(I-J_2VJ_{11}V^{*})
\end{pmatrix}.
\end{equation}

Now the assumption \eqref{crit} will be used again. Since
$\nu_-[I-T_{1}^{[*]}T_{1}]=\nu_-[I-T_{11}^2]-\nu_-(J_2)$ and
$T_{21}=J_2VD_{T_{11}}$ it follows from Theorem \ref{thmB} that
$V^*\in[\sH_2,\sD_{T_{11}}]$ is $J$-contractive:
$J_2-VJ_{11}V^*\ge0$. Therefore, \eqref{E:6} shows that
$T_M\ge_{J_2} T_m$ and $I+T_M\ge_{J_2} I+T_m$ and hence, in addition to
$I+T_m$, also $I+T_M$ is a solution to the problem $A_{+}^0$ and, in
particular, $\nu_-[I+T_M]=\kappa_-=\nu_-[I+T_m]$. Similarly,
$I-T_M\le_{J_2} I-T_m$ which implies that $I-T_m$ is also a solution to
the problem $A_{-}^0$, in particular,
$\nu_-[I-T_m]=\kappa_+=\nu_-[I-T_M]$. Now by applying Lemma
\ref{L:wT} we get
$$
\nu_-[I-T_m^2]=\kappa-\nu_-(J_2),
$$
$$
\nu_-[I-T_M^2]=\kappa-\nu_-(J_2).
$$
Therefore, $T_m,T_M\in \Ext_{T_1,\kappa}(-1,1)_{J_2}$ which in particular
proves that the condition \eqref{crit} is sufficient for solvability
of the completion problem \eqref{E:A}.

(iv) Observe, that $T\in\Ext_{T_1,\kappa}(-1,1)_{J_2}$ if and only if
$T=T^{[*]}\supset T_1$ and $\nu_-[I\pm T]=\kappa_\mp$. By Theorem
\ref{T:1} this is equivalent to
\begin{equation}\label{E:7}
J_2S_-^{*}J_-S_--I\le_{J_2} T_{22}\le_{J_2} I-J_2S_+^{*}J_+S_+.
\end{equation}
The inequalities \eqref{E:7} are equivalent to \eqref{E:Ext}.

(v) The relations \eqref{E:5} follow from \eqref{E:S pm} and
\eqref{E:T}.
\end{proof}
\end{subsection}
\bigskip

\noindent \textbf{Acknowledgements.} The author thanks his supervisor Seppo Hassi for several detailed discussions on the results of this paper.


\begin{thebibliography}{99} 


%



\bibitem{ACS2006}
Antezana,~J., Corach,~G., and Stojanoff, D., Bilateral shorted
operators and parallel sums. Linear Algebra Appl. \textbf{414}
(2006), 570--588.

%
%






\bibitem{AG82}
Arsene, Gr. and Gheondea, A., Completing Matrix Contractions, J.
Operator Theory, \textbf{7} (1982), 179--189.

\bibitem{ACG87}
Arsene, Gr., Constantinescu, T., Gheondea, A., Lifting of Operators
and Prescribed Numbers of Negative Squares, Michigan Math. J.,
\textbf{34} (1987), 201--216.


\bibitem{AIbook}
Azizov, T.Ya. and Iokhvidov, I.S., \textit{Linear operators in
spaces with indefinite metric}, John Wiley and Sons, New York, 1989.


\bibitem{BH2015}
Baidiuk, D., and Hassi, S., Completion, extension, factorization,
and lifting of operators, Arxiv 2014, (to apper in Math. Ann.)

%
%

\bibitem{Bognarbook}
Bogn\'ar, J., \textit{Indefinite Inner Product Space},
Springer-Verlag, Berlin, 1974.

%
%
%

\bibitem{CG89}
Constantinescu, T. and Gheondea, A.: Minimal Signature of Lifting
operators. I, J. Operator Theory, \textbf{22} (1989), 345--367.

\bibitem{CG92}
Constantinescu, T. and Gheondea, A.: Minimal Signature of Lifting
operators. II, J. Funct. Anal., \textbf{103} (1992), 317--352.

\bibitem{DaKaWe}
Davis, Ch., Kahan, W.M., and Weinberger, H.F., Norm preserving
dilations and their applications to optimal error bounds, SIAM J.
Numer. Anal., \textbf{19}, no. 3 (1982), 445--469.

%

%
%
%
%
%

\bibitem{Drit90}
Dritschel, M.A., A lifting theorem for bicontractions on Kre\u{\i}n
spaces, J. Funct. Anal., \textbf{89} (1990), 61--89.


\bibitem{DritRov90}
Dritschel, M.A. and Rovnyak, J., Extension theorems for contraction
operators on Kre\u{\i}n spaces. \textit{Extension and interpolation
of linear operators and matrix functions}, 221--305, Oper. Theory
Adv. Appl., \textbf{47}, Birkh\"auser, Basel, 1990.

%
%


%

\bibitem{HMS04}
Hassi,~S., Malamud,~M.M., and de~Snoo,~H.S.V., On Kre\u{\i}n's
Extension Theory of Nonnegative Operators, Math. Nachr.,
\textbf{274/275} (2004), 40--73.

%

\bibitem{Kato}
Kato,~T., \textit{Perturbation Theory for Linear Operators},
Springer-Verlag, Berlin, Heidelberg, 1995.


\bibitem{KM1}
Kolmanovich, V.U. and Malamud, M.M., Extensions of Sectorial
operators and dual pair of contractions, (Russian) {Manuscript
No~4428-85. Deposited at Vses. Nauchn-Issled, Inst. Nauchno-Techn.
Informatsii, VINITI 19 04 85, Moscow}, R ZH Mat 10B1144, (1985),
1--57.

\bibitem{Kr44}
Kre\u{\i}n,~M.G., On hermitian operators with defect indices
$(1,1)$, Dokl. Akad. Nauk SSSR, \textbf{43} (1944), 339--342.

\bibitem{Kr46}
Kre\u{\i}n,~M.G., On resolvents of Hermitian operator with
deficiency index $(m,m)$, Dokl. Akad. Nauk SSSR, \textbf{52} (1946),
657--660.

\bibitem{Kr47}
Kre\u{\i}n,~M.G., Theory of Selfadjoint Extensions of Semibounded
Operators and Its Applications, I, Mat. Sb. \textbf{20}, No.3
(1947), 431--498.

%
%
%

\bibitem{LT87}
Langer,~H. and Textorius,~B., Extensions of a bounded Hermitian
operator $T$ preserving the numbers of negative eigenvalues of
$I-T^*T$, Research Report LiTH-MAT-R-87-17, Department of
Mathematics, Link\"oping University, (1977), 15 pp.


%
%
%
%

\bibitem{S59}
Shmul'yan, Yu. L., A Hellinger operator integral, (Russian) Mat. Sb.
(N.S.), \textbf{49}, No.91 (1959), 381--430.

%

\bibitem{ShYa}
Shmul'yan, Yu. L. and Yanovskaya,~R.N., Blocks of a contractive
operator matrix, Izv. Vyssh. Uchebn. Zaved. Mat.,  No. 7 (1981),
72--75.



\end{thebibliography}
\end{document}